\documentclass[reqno]{amsart}

\usepackage{amssymb,amsfonts,amstext,amsthm,hyperref,cleveref,xcolor}

\usepackage{graphics}
\usepackage{graphicx}
\usepackage{epstopdf}
\usepackage{indentfirst}
\usepackage{float}
\usepackage[numbers,sort&compress]{natbib}
\bibpunct[, ]{[}{]}{,}{n}{,}{,}
\makeatletter
\def\NAT@def@citea{\def\@citea{\NAT@separator}}
\makeatother
\theoremstyle{plain}
\newtheorem{theorem}{Theorem}[section]
\newtheorem{lemma}[theorem]{Lemma}

\newtheorem{claim}[theorem]{Claim}

\crefrangeformat{equation}{#3(#1)#4--#5(#2)#6}
\crefname{enumi}{\unskip}{\unskip}

\theoremstyle{definition}

\newtheorem{remark}[theorem]{Remark}

\begin{document}

\title[Coninvolutions over Gaussian Integers and Quaternions integers modulo $p$]{Coninvolutions on Upper Triangular Matrix Group over the Ring of Gaussian Integers and Quaternions integers modulo $p$
}

\author{Ivan Gargate}
\address{UTFPR, Campus Pato Branco, Rua Via do Conhecimento km 01, 85503-390 Pato Branco, PR, Brazil}
\email{ivangargate@utfpr.edu.br}

\author{Michael Gargate}
\address{UTFPR, Campus Pato Branco, Rua Via do Conhecimento km 01, 85503-390 Pato Branco, PR, Brazil}
\email{michaelgargate@utfpr.edu.br}
\begin{abstract}
In this article we give various formulates for compute the number of all coninvolutions over the group of upper triangular matrix with entries into the ring of Gaussian integers module $p$ and the ring of Quaternions integers module $p$, with $p$ an odd prime number. 
\end{abstract}
\maketitle
\section{Introduction}\label{intro}
Let $R$ be an finite ring endowed with a compatible complex conjugate structure. Denote by $T_n(R)$ the $n\times n$ upper triangular matrix group with entries in $R$. A matrix $A\in T_n(R)$ is called an Involution iff $A^2=I_n$ with $I_n$ the identity in $T_n(R)$. Involutions on $T_n(R)$ has been studied by several authors, see for instance \cite{Slowik1, Hou} also over Incidence Algebras see \cite{Ivan}. On the other hand, Coninvolution matrices has been studied by \cite{Ikramov} and \cite{Horn}. In the present article we study how compute the number of Coninvolution matrices over  two specially finite rings, the ring of Gaussian Integers module $p$ and the ring of Quaternions Integers module $p$, this is  motivated by the work initiated by Slowik in \cite{Slowik1}.  


An matrix $A=(a_{rs})\in T_n(R)$ is called a Coninvolution iff $A\bar{A}=I_n$, where $\bar{A}=(\bar{a}_{rs})$ and $\bar{a}$ denote the complex conjugate of the number $a$. Denote by $\mathcal{CI}(n,R)$ the number of all coninvolution matrices there are in $T_n(R)$. 

In the section 2 we show how construct a Coninvolution matrix and with this algorithm we compute the number of all Coninvonlutions with $R=\mathbb{Z}_p[i]$, the ring of Gaussian Integers module $p$, i.e, we calculate the number  $\mathcal{CI}(n,\mathbb{Z}_p[i])$.

In the section 3, similarly, we show how construct a Coninvolution matrix but in the block upper triangular matrix group $T_{n}(M_s(\mathbb{Z}_p)$ that is isomorphic to the group $T_n(\mathbb{Z}_p[i,j,k])$ whose entries are in the Ring of Quaternion Integers module $p$. With this isomorphism its possible describe all Coninvolutions in $T_n(\mathbb{Z}_p[i,j,k])$. With this result, we finally compute the number of all Coninvolutions $\mathcal{CI}(n,\mathbb{Z}_p[i,j,k])$. In both cases we assume that $p$ is an odd prime number.

In the section 4 we present various Tables of the numbers of coninvolutions in $T_n(\mathbb{Z}_p[i])$ and $T_n(\mathbb{Z}_p[i,j,k])$.

Our main results are the followings:
\begin{theorem}\label{teo2}
Let $\mathbb{Z}_p[i]$ the Gaussian integers module $p$ . Then the number of Coninvoltions in the group $T_n(\mathbb{Z}_p[i])$ is equal to
$$ \mathcal{CI}(n,\mathbb{Z}_p[i])=|U(\mathbb{Z}_p[i])|^n\times p^{\frac{(n-1)n}{2}},$$
where $|U(\mathbb{Z}_p[i])|$ denote the cardinality of $U(\mathbb{Z}_p[i])=\{z \in \mathbb{Z}_p[i], \  |z|^2=1\}.$
\end{theorem}

And, for the ring of Quaternion Integers module $p$ we have the following Theorem.

\begin{theorem}\label{th3} Consider $s=|SL(\mathbb{Z}_p)|=(p^2-1)p$. Then the number of coinvolutions over $T_{n}(\mathbb{Z}_p[i,j,k])$, denote by $\mathcal{CI}(n,\mathbb{Z}_p[i,j,k])$, is equal to
$$\! \sum^{min\{s,n\}}_{j=1}\!\!\!\!\!\sum_{\substack{n_1\!+\!n_2\!+\!\cdots\!+\!n_j\!=\!n\! \\ n_1\leq n_2 \leq \cdots \leq  n_j}}\frac{s!}{(s\!-\!j)!g(n_1,n_2,\!\cdots\!,n_j)}\!\!\!\binom{n}{n_1n_2\cdots n_s}  p^{\frac{1}{2}[\displaystyle n^2\!-\!\sum^j_{u=1}n_u^2\!+\!3n_u(n_u\!-\!1)]},$$
where
$g(n_1,\cdots,n_j)=r_1!\cdot r_2! \cdots r_t!$ if and only if $m_1=m_2=\cdots=m_{r_1}\neq m_{r_1+1}=m_{r_1+2}=\cdots = m_{r_1+r_2}\neq m_{r_1+r_2+1}=\cdots$
\end{theorem}

\section{Coninvolutions over the Gaussian Integers module $p$}
Consider $p>2$ a prime number. Let $\mathbb{Z}_p[i]=\{a+ib,\ a,b \in \mathbb{Z}_p \ and \ i^2=-1\}$ be the Gaussian integers module  $p$, then if  $z=a+ib \in \mathbb{Z}_p[i]$ denote by $Re(z)=a$ and $Im(z)=b$ the real and imaginary parts of $z$, respectively. Also, define the natural complex conjugation $\bar{z}=a+(p-1)ib=a-ib$ and $|z|^2=z\cdot \bar{z}=a^2+b^2$ the natural  modulus. For a matrix $A=(z_{rs})\in T_n(\mathbb{Z}_p[i])$ define their complex conjugation as $\bar{A}=(\bar{z}_{rs})$. A matrix $A\in T_n(\mathbb{Z}_p[i])$  is called a coninvolution if $A\bar{A}=I_n$ where $I_n$ is the identity matrix.  Simillarly to \cite{Slowik1} we have the following theorem:

\begin{theorem}\label{teo1} Let $\mathbb{Z}_p[i]$ the ring of Gaussian integers module $p$. A matrix $A=(z_{rs}) \in T_n(\mathbb{Z}_p[i])$ is a convinvolution if and only if $A$ is described by the following statements:
\begin{itemize}
    \item [(i)] For all $z_{rr}$ with $1\leq r \leq n$ we have $|z_{rr}|^2=1$.
    
    \item[(ii)] If $z_{rr}=\bar{z}_{ss}$ then 
    $$z_{rs}=\left\{\begin{array}{ll}i\gamma &, \ if \ s=r+1 \\ -\displaystyle\frac{(z_{rr})^{-1}}{2}\cdot\displaystyle\sum^{s-1}_{t=r+1}z_{rt}\bar{z}_{ts}+i\gamma &,\ if \ s>r+1,\end{array}\right.$$  
    \item[(iii)] If $z_{rr}=-\bar{z}_{ss}$ then
    $$z_{rs}=\left\{\begin{array}{ll}\gamma &, \ if \ s=r+1 \\ \gamma-\displaystyle\frac{(z_{rr})^{-1}}{2}\cdot \displaystyle \sum^{s-1}_{t=r+1}z_{rt}\bar{z}_{ts} &, \ if \ s>r+1,  \end{array}\right.$$
    \item[(iv)] If $z_{rr}=z_{ss}$ then
    \item[a.)] If $Re(z_{rr}) \neq 0$ we have
    $$z_{rs}=\left\{\begin{array}{ll}- [Re(z_{rr})^{-1}Im(z_{rr})-i]\gamma&, \ if \ s=r+1 \\
    -(2\ Re(z_{rr}))^{-1}\cdot \displaystyle\sum^{s-1}_{t=r+1}z_{rt}\bar{z}_{ts}- [Re(z_{rr})^{-1}Im(z_{rr})-i]\gamma \!\!\!&,\ if \ s>r+1\end{array}\right.$$
    \item[b.)] If $Im(z_{rr})\neq 0$ we have
    $$z_{rs}=\left\{\begin{array}{ll}[1-i Im(z_{rr})^{-1} Re(z_{rr})]\gamma \!\!\!&,\ if \ s=r+1 \\ 
    (1-i Im(z_{rr})^{-1} Re(z_{rr}))\gamma - i(2 Im(z_{rr}))^{-1}\cdot \displaystyle\sum^{s-1}_{t=r+1}z_{rt}\bar{z}_{ts}&,\ if \ s>r+1,\end{array}\right.$$
    \item[(v)] If $z_{rr}=-z_{ss}$ then
    \item[a.)] If $Re(z_{rr})\neq 0$ we have
    $$z_{rs}=\left\{\begin{array}{ll}[1+i\ Re(z_{rr})^{-1} Im(z_{rr})]\gamma\!\!\!&\!\!\!,\ if \ s=r+1 \\ 
    (2\ Re(z_{rr}))^{-1}\cdot ( \displaystyle\sum^{s-1}_{t=r+1}z_{rt}\bar{z}_{ts})+[1+i\ Re(z_{rr})^{-1} Im(z_{rr})]\gamma\!\!\!&\!\!\!, \ if \ s>r+1,\end{array}\right.$$
    \item[b.)] If $Im(z_{rr})\neq 0 $ we have
    $$z_{rs}=\left\{\begin{array}{ll}[(Im(z_{rr})^{-1}\ Re(z_{rr})+i] \gamma\!\!\!&\!\!\!,\ if \ s=r+1 \\ 
    i(2\ Im(z_{rr}))^{-1}\cdot ( \displaystyle\sum^{s-1}_{t=r+1}z_{rt}\bar{z}_{ts})+   [(Im(z_{rr})^{-1}\ Re(z_{rr})+i] \gamma\!\!\!&\!\!\!,\ if \ s>r+1,
     \end{array}\right.$$
    \item[(vi)] If $z_{rr}\neq \pm z_{ss}$ and $z_{rr}\neq \pm \bar{z}_{ss}$ then
    $$z_{rs}=\left\{\begin{array}{ll}[- (Re(z_{rr}+z_{ss}))^{-1}\cdot Im(z_{rr}+z_{ss}
    )+i]\gamma\!\!\!&\!\!\!,\ if \ s=r+1 \\ 
    Re(z_{rr}+z_{ss})^{-1}  Re(-\displaystyle\sum^{s-1}_{t=r+1}z_{rt}\bar{z}_{ts})+ & \\
    \ \ \ \ \ \ \ \ \ \  \ \ \ \ \ \ \ \ +[- (Re(z_{rr}+z_{ss}))^{-1}\cdot Im(z_{rr}+z_{ss}
    )+i]\gamma\!\!\!&\!\!\!,\ if \ s>r+1,
    \end{array}\right.$$
    where $\gamma\in \mathbb{Z}_p$, in all cases, may be arbitrary.
    \end{itemize}
\end{theorem}

\begin{proof} For $r\neq s$ consider $z_{rs}=x+iy$ with $x,y \in \mathbb{Z}_p$.
\begin{itemize} 
    \item [(i)] If $A=(z_{rs})$ is a coninvolution then $z_{rr}\bar{z}_{rr}=|z_{rr}|^2=1$ for all $r$. Also, we have if $s>r+1$, the following equation
    $$z_{rr}\bar{z}_{rs}+z_{r(r+1)}\bar{z}_{(r+1)s}+ \cdots + z_{r(s-1)}\bar{z}_{(s-1)s}+z_{rs}\bar{z}_{ss}=0.$$
or
\begin{equation}\label{eq1}    
z_{rr}\cdot \bar{z}_{rs}+z_{rs}\cdot \bar{z}_{ss}= - \displaystyle \sum_{p=r+1}^{s-1} z_{rp}\cdot \bar{z}_{ps}.
\end{equation}    
    And, if $s=r+1$ we have only the equation
    $$z_{rr}\cdot\bar{z}_{r,r+1}+z_{r,r+1}\cdot \bar{z}_{r+1,r+1}=0.$$
    So, we analise only the case if $s>r+1$.
    \item[(ii)] In this case, 
    $$z_{rr}\cdot \bar{z}_{rs}+z_{rs}\cdot \bar{z}_{ss}=z_{rr}\cdot      Re(z_{rs})= - \displaystyle \sum_{t=r+1}^{s-1} z_{rt}\cdot \bar{z}_{ts},$$ then
    $$x=Re(z_{rs})=(z_{rr})^{-1}\cdot (- \displaystyle \sum_{t=r+1}^{s-1} z_{rt}\cdot \bar{z}_{ts}),$$
    here $z_{rs}=x+iy$ with $y \in \mathbb{Z}_p$ may be arbitrary.
    \item[(iii)] Is similar that (ii).
    \item[(iv)] If $z_{rr}=z_{ss}$, denote by $z_{rr}=a+ib$ and $z_{ss}=\alpha+i \beta$, with $a,b,\alpha,\beta \in \mathbb{Z}_p$, then the equation (\ref{eq1}) is expressed in the form:
\begin{equation}\label{eq2}
    - \displaystyle \sum_{t=r+1}^{s-1} z_{rt}\cdot \bar{z}_{ts}= \left[(a+\alpha)x+(b+\beta)y\right]+ i \left[(b-\beta)x+(\alpha-a)y\right].
\end{equation}
    In the case that $a\neq 0$ then we have 
    $$x=(2a)^{-1}\left\{- \displaystyle \sum_{t=r+1}^{s-1} z_{rt}\cdot \bar{z}_{ts}-2by\right\},$$
    so
    $$\begin{array}{ll}z_{rs}&= x+iy=(2a)^{-1}\left\{- \displaystyle \sum_{t=r+1}^{s-1} z_{rt}\cdot \bar{z}_{ts}-2by\right\}+iy \\ & = (2a)^{-1} \left\{- \displaystyle \sum_{t=r+1}^{s-1} z_{rt}\cdot \bar{z}_{ts}\right\}- \left\{a^{-1}b-i\right\}y,\end{array} $$
    with $y \in \mathbb{Z}_p$ may be arbitrary.
    The case $b\neq 0$ is similar.
    
    \item[(v)] If $z_{rr}=-z_{ss}$ and $a\neq 0$ then follows from the equation (\ref{eq2}) we have that
    $$y=(-2ai)^{-1}\left\{- \displaystyle \sum_{t=r+1}^{s-1} z_{rt}\cdot \bar{z}_{ts}-2bi x\right\},$$
    so
    $$\begin{array}{ll}z_{rs}&= x+iy=x+(-2ai)^{-1}\left\{- \displaystyle \sum_{t=r+1}^{s-1} z_{rt}\cdot \bar{z}_{ts}-2bi x\right\}i \\ &= (-2a)^{-1}\left\{- \displaystyle \sum_{t=r+1}^{s-1} z_{rt}\cdot \bar{z}_{ts}\right\}+\left\{1+ia^{-1}b \right\}x,\end{array} $$
with $x \in \mathbb{Z}_p$ may be arbitrary. The case $b\neq 0$ is similar.    
 \item[(vi)] The equation (\ref{eq2}) can be write in the following linear system
 \begin{eqnarray}
     (a+\alpha)x+(b+\beta)y = Re(-  \sum_{t=r+1}^{s-1} z_{rt}\cdot \bar{z}_{ts}), \label{equat1} \\ (b-\beta)x+(\alpha-a)y = Im(-  \sum_{t=r+1}^{s-1} z_{rt}\cdot \bar{z}_{ts}). \label{equat2}
 \end{eqnarray}
 Here, of the equation (\ref{equat1}) we have that
 $$x=(a+\alpha)^{-1}\cdot [ Re(-  \sum_{t=r+1}^{s-1} z_{rt}\cdot \bar{z}_{ts})-(b+\beta)y ] $$
 and substituting into the equation (\ref{equat2}) we have $$(b-\beta)\cdot (a+\alpha)^{-1}\cdot [ Re(- \displaystyle \sum_{t=r+1}^{s-1} z_{rt}\cdot \bar{z}_{ts})-(b+\beta)y ]+(\alpha-a)y=Im(- \displaystyle \sum_{t=r+1}^{s-1} z_{rt}\cdot \bar{z}_{ts}),$$
then
$$
\begin{array}{lc}
    y \left\{(\alpha-a)-(b^2-\beta^2)(a+\alpha)^{-1}\right\}=  \\
 \ \ \ \ \ \ \ \   \ \ \ \ \ \ \ \ \ \ \ \ =Im(- \displaystyle \sum_{t=r+1}^{s-1} z_{rt}\cdot \bar{z}_{ts})-(b-\beta)(a+\alpha)^{-1}\cdot Re(- \displaystyle \sum_{t=r+1}^{s-1} z_{rt}\cdot \bar{z}_{ts}).
\end{array}
$$
Notice that, by the item (i) we have that $a^2+b^2=\alpha^2+\beta^2=1$ and here $b^2-\beta^2=\alpha^2-a^2=(\alpha-a)(\alpha+a)$ 
then $(\alpha-a)=(b^2-\beta^2)(a+\alpha)^{-1}$. Follow that the second side of the last identity is null and also, $(\alpha-a)-(b^2-\beta^2)(a+\alpha)^{-1}=0$, so $y\in \mathbb{Z}_p$ can be arbitrary. Finally, we have
$$\begin{array}{rl}z_{rs}=&x+iy=(a+\alpha)^{-1}\cdot [ Re(- \displaystyle \sum_{t=r+1}^{s-1} z_{rt}\cdot \bar{z}_{ts})-(b+\beta)y ]+iy \\ =& (a+\alpha)^{-1}\cdot Re(- \displaystyle \sum_{t=r+1}^{s-1} z_{rt}\cdot \bar{z}_{ts}) + \left\{-(a+\alpha)^{-1}\cdot(b+\beta)+i\right\}y. \end{array}$$
\end{itemize}
\end{proof}

Then, with this, we show the Theorem \ref{teo2}:
\begin{proof}[Proof of Theorem \ref{teo2}] Follow immediately from Theorem \ref{teo1} that, if $A$ is a Coninvolution in $T_n(\mathbb{Z}_p[i])$ then over the main diagonal we can choose any element in $U(\mathbb{Z}_p[i])$ and independent from these choices all entries over the main diagonal can be choose that depending from one variable. How we can $\frac{(n-1)n}{2}$ entries over the main diagonal then we conclude that we are $$|U(\mathbb{Z}_p[i])|^n\times p^{\frac{(n-1)n}{2}}$$ and this conclude the proof.
\end{proof}


\section{Coninvolutions over the Quaternion Integers module $p$}
Consider the set $\{i,j,k\}$ such that satisfies the relations $i^2=j^2=k^2=-1$ and $ij=-ji=k$, and define the set
$$\mathbb{Z}_p[i,j,k]=\{z=x_0+x_1i+x_2j+x_3k, \ x_0,x_1,x_2,x_3\in \mathbb{Z}_p\},$$
with natural operations of sum and product. The set $\mathbb{Z}_p[i,j,k]$ is called the ring of Quaternions Integers module $p$. We define the conjugation of the number $z$ as $$\bar{z}=x_0+(p-1)x_1i+(p-1)x_2 j+(p-1)x_3 k=x_0-x_1i-x_2j-x_3k.$$ 

If $A=(A_{rs})\in T_n(\mathbb{Z}_p[i,j,k])$ is an upper triangular matrix, then define the complex conjugation of $A$ as $\bar{A}=(\bar{A}_{rs}).$ An matrix $A\in T_n(\mathbb{Z}_p[i,j,k])$ is called an coninvolution if $A\bar{A}=I_n$ where $I_n$ is the identity in $T_n(\mathbb{Z}_p[i,j,k])$. In order to calculate the number of coninvolutions on $T_n(\mathbb{Z}_p[i,j,k])$ we consider the isomorphims multiplicative $\varphi : \mathbb{Z}_p[i,j,k] \to M_2(\mathbb{Z}_p)$
define by
$$\varphi(z)=\varphi(x_0+x_1i+x_2j+x_3k)=$$
$$=x_0\left[\begin{array}{cc}1 & 0 \\ 0 & 1\end{array}\right]+x_1\left[\begin{array}{cc}0 & 1 \\ p-1 & 0\end{array}\right]+x_2\left[\begin{array}{cc}a & b \\ b & p-a \end{array}\right]+x_3\left[\begin{array}{cc}b & p-a \\ p-a & p-b\end{array}\right],$$
where $a,b\in \mathbb{Z}_p$ such that $a^2+b^2=p-1$. These numbers exists if $p$ is an odd  prime number and in this case $\varphi$ is an isomorphism multiplicative (see \cite{Miguel}). 
For $A=\left[\begin{array}{cc}
M_1 & M_2 \\ M_3 & M_4 
\end{array}\right] \in M_2(\mathbb{Z}_p)$ define the $\varphi$-conjugation of $A$ and denote by $\widetilde{A}$ to the matrix $\left[\begin{array}{cc}M_4 & -M_2 \\ -M_3 & M_1\end{array}\right].$
If $\varphi(z)=A$ then is not difficult show that $\varphi(\bar{z})=\widetilde{A}$.  Using this isomorphism we can enunciated the following result.
\begin{lemma}\label{lemma1}
Let $z \in \mathbb{Z}_p[i,j,k]$ and consider $\varphi(z)=A\in M_2(\mathbb{Z}_p)$ their respective representation. Then
$z\cdot \bar{z}=1$ if and only if $\det (A)=1$. So, there are $(p^2-1)p$ elements in $\mathbb{Z}_p[i,j,k]$ that satisfies the equation $z\bar{z}=1$.
\end{lemma}
\begin{proof}
By the isomorpshim $\varphi$ this is equivalent to proof that $A\widetilde{A}=I$. If $A=\left[\begin{array}{cc}M_1 & M_2 \\ M_3 & M_4\end{array}\right]$ then the above equation is true if $\det(A)=M_1M_4-M_2M_3=1$. We conclude the proof by the observation that  $|SL_2(\mathbb{Z}_p)|=(p^2-1)p.$ 
\end{proof}

 We consider the block upper triangular matrix group $T_n(M_2(\mathbb{Z}_p))$ where the entries are $2\times 2$ matrices with entries in $\mathbb{Z}_p$. In this group, $M=(M_{rs})\in T_n(M_2(\mathbb{Z}_p))$ is called an coninvolution if $M\widetilde{M}=I$ where $\widetilde{M}=(\widetilde{M}_{rs}).$ The isomorphism multiplicative $\varphi$ can be extended naturally to an isomorphism multiplicative into the groups $T_n(\mathbb{Z}_p[i,j,k])$ and $T_n(M_2(\mathbb{Z}_p))$, so, by the isomorphism $\varphi$, we can conclude that  $A=(a_{rs})$ is a coninvolution in $T_n(\mathbb{Z}_p[i,j,k])$ if and only if $M=(\varphi(a_{rs}))$ is a coninvolution in $T_n(M_2(\mathbb{Z}_p))$. The following Theorem study how we can construct a coninvolution in $T_n(M_2(\mathbb{Z}_p))$.

\begin{theorem}\label{teo4} A block upper triangular matrix $M=(Z_{rs})\in T_n(M_2(\mathbb{Z}_p))$ is a convinvolution if and only if $M$ described by the following statements:

\begin{itemize}
    \item [1.)] For all $Z_{rr}$ we have $Z_{rr}\widetilde{Z}_{rr}=I_2$ and we can conclude that $\det(Z_{rr})=1$.
    \item[2.)] Denote by $Z_{rr}=\left[\begin{array}{cc}
         a & b \\ c & d\end{array}\right]$, $Z_{ss}=\left[\begin{array}{cc}x \!&\! y \\ w\! &\! z \end{array}\right]$ and in the case that $s>r+1$ denote by $-\displaystyle \sum^{s-1}_{t=t+1}Z_{rt}\widetilde{Z}_{ts}=\left[\begin{array}{cc}A & B \\ C & D\end{array}\right]$ and $\theta=det(Z_{ss}-Z_{rr})$. Then, for all $Z_{r,s}$ we have
    \item[a.)] If $Z_{rr}=Z_{ss}$, then
    \begin{itemize}
        \item [(i)] If $=a\neq 0$ then $Z_{rs}$ is equals to:
        $$\left\{\!\!\begin{array}{*{20}{ll}}\!\beta_1\!\left[\begin{array}{*{10}{cc}}\!1\! &\! 0\! \\ \!0\! &\! -a^{-1}d\!\end{array}\right]\! + \!\beta_2\!\left[\begin{array}{cc}0 & 1 \\ 0 & a^{-1}c\end{array}\right]\!+\!\beta_3\!\left[\begin{array}{cc}0 & 0 \\ 1 & a^{-1}b\end{array}\right]\!\!\!&\!\!\! ,\ if \ \!s\!=\!r\!+\!1 \\
        &\\
        \left[\begin{array}{cc}0 & 0 \\ 0 & a^{-1}A\end{array}\right]\!+\!\beta_1\!\left[\begin{array}{cc}1 & 0 \\ 0 & -a^{-1}d\end{array}\right]\!+\!\beta_2\!\left[\begin{array}{cc}0 & 1 \\ 0 & a^{-1}c\end{array}\right]\!+\!\beta_3\!\left[\begin{array}{cc}0 & 0 \\ 1 & a^{-1}b\end{array}\right]\!\!\!&\!\!\!,\ if \!\ \!  s\!>\!r\!+\!1\end{array}\right.$$
        
        \item[(ii)] If $b\neq 0$ then $Z_{rs}$ is equal to:
        $$\!\left\{\!\!\begin{array}{ll}\!\beta_1\!\left[\begin{array}{cc}1 & 0 \\ b^{-1}d & 0  \end{array}\right]\!+\!\beta_2\!\left[\begin{array}{cc}0 & 1 \\ -b^{-1}c&0 \end{array}\right]\!+\!\beta_3\!\left[\begin{array}{cc}0 & 0 \\ b^{-1}a & 1\end{array}\right]\!\!\!&\!\!\!,\ if \!\ \! s\!=\!r\!+\!1\! \\
        &\\
        \!\left[\begin{array}{cc}0 & 0 \\ -b^{-1}A &  \end{array}\right]\!+\!\beta_1\!\left[\begin{array}{cc}1 & 0 \\ b^{-1}d & 0  \end{array}\right]\!+\!\beta_2\!\left[\begin{array}{cc}0 & 1 \\ -b^{-1}c&0 \end{array}\right]\!+\!\beta_3\!\left[\begin{array}{cc}0 & 0 \\ b^{-1}a & 1\end{array}\right]\!\!\!&\!\!\!,\ if\!\ \!s\!>\!r\!+\!1\!
        \end{array}\right.$$
        
        \item[(iii)] If $c\neq 0$ then $Z_{rs}$ is equal to:
        $$\left\{\!\begin{array}{ll}\!\beta_1\!\left[\begin{array}{cc}1 & c^{-1}d \\ 0 & 0\end{array}\right]\!+\!\beta_2\!\left[\begin{array}{cc}0 & -c^{-1}b \\ 1 & 0\end{array}\right]\!+\!\beta_3\!\left[\begin{array}{cc}0 & c^{-1}a \\ 0 & 1\end{array}\right]\!\!\! &\!\!\!,\ if\! \ \! s\!=\!r\!+\!1\! \\
        &\\
        \!\left[\begin{array}{cc}0 & -c^{-1}A \\ 0 & 0\end{array}\right]\!+\!\beta_1\!\left[\begin{array}{cc}1 & c^{-1}d \\ 0 & 0\end{array}\right]\!+\!\beta_2\!\left[\begin{array}{cc}0 & -c^{-1}b \\ 1 & 0\end{array}\right]\!+\!\beta_3\!\left[\begin{array}{cc}0 & c^{-1}a \\ 0 & 1\end{array}\right]\!\!\!&\!\!\!,\ if\! \ \!s\!>\!r\!+\!1\!\end{array}\right.$$
        \item[(iv)] If $d\neq 0$
       then $Z_{rs}$ is equal to:
       $$\!\left\{\!\begin{array}{ll}\!\beta_1\!\left[\begin{array}{cc}d^{-1}c& 1 \\ 0 & 0\end{array}\right]\!+\!\beta_2\!\left[\begin{array}{cc}d^{-1}b& 0 \\ 1 & 0\end{array}\right]\!+\!\beta_3\!\left[\begin{array}{cc}-d^{-1}a& 0 \\ 0 & 1\end{array}\right]\!\!\!&\!\!\!,\ if \!\ \!s\!=\!r\!+\!1\! \\
       &\\
       \!\left[\begin{array}{cc}d^{-1}A& 0 \\ 0 & 0\end{array}\right]\!+\!\beta_1\!\left[\begin{array}{cc}d^{-1}c& 1 \\ 0 & 0\end{array}\right]\!+\!\beta_2\!\left[\begin{array}{cc}d^{-1}b& 0 \\ 1 & 0\end{array}\right]\!+\!\beta_3\!\left[\begin{array}{cc}-d^{-1}a& 0 \\ 0 & 1\end{array}\right]\!\!\!&\!\!\!,\ if\!\ \!s\!>\!r\!+\!1\!\end{array}\right.,$$
       where $\beta_1,\beta_2,\beta_3\in \mathbb{Z}_p$ are arbitrary elements, so, depending from three variables.
    \end{itemize}
    \item[b.)] If $Z_{rr}\neq Z_{ss}$ then we consider the following cases
    \begin{itemize}
        \item [(i)] If $a=x,b\neq y, c \neq w$ and $d\neq z$ then $Z_{rs}$ is equal to
        $$\left\{\begin{array}{ll}\gamma\left[\begin{array}{cc} 0 & (z-d)^{-1}(y-b) \\ -(z-d)^{-1}(c-w) & 1\end{array}\right]&,\ if \ s=r+1 \\
        &\\
        \left[\begin{array}{cc} (b-y)^{-1}B & (c-w)^{-1}W-(z-d)^{-1}(c-w)^{-1}(y-b)C \\ (z-d)^{-1}C & 0\end{array}\right] &, \ if \ s>r+1, \\
        \ \ \ \ \ \ \ \ \ \ +\gamma\left[\begin{array}{cc} 0 & (z-d)^{-1}(y-b) \\ -(z-d)^{-1}(c-w) & 1\end{array}\right] & \end{array}\right.$$
        where $W=A-D-(z-d)(b-y)^{-1}B$. 
        
        \item[(ii)] If $b=y, a\neq x, c\neq w$ and $d\neq z$ then $Z_{rs}$ is equal to:
        $$\left\{\begin{array}{ll} \gamma\left[\begin{array}{cc}(c-w)^{-1}(a-x) & 0 \\ 1 & -(c-w)^{-1}(z-d)\end{array}\right] &, \ if \ s=r+1 \\ 
        &\\
        \left[\begin{array}{cc}(z-d)^{-1}W-(z-d)^{-1}(c-w)^{-1}(a-x)C & (x-a)^{-1}B \\ 0 & (c-w)^{-1}C\end{array}\right] &, \ if \ s>r+1 \\
        \ \ \ \ \ \ \ \ \ \ +\gamma\left[\begin{array}{cc}(c-w)^{-1}(a-x) & 0 \\ 1 & -(c-w)^{-1}(z-d)\end{array}\right]& 
        \end{array}\right.$$
         where $W=A-D-(c-w)(x-a)^{-1}B$.
        \item[(iii)] If $c=w,a\neq x,b\neq y$ and $d\neq z$ then $Z_{rs}$ is equal to:
        $$\left\{\begin{array}{ll}\gamma\left[\begin{array}{cc}-(b-y)^{-1}(x-a) & 1 \\ 0 & -(b-y)^{-1}(z-d)\end{array}\right] &, \ if \ s=r+1\\
        &\\
        \left[\begin{array}{cc}(b-y)^{-1}B & 0 \\ (z-d)^{-1}C & (a-x)^{-1}W-(b-y)^{-1}(a-x)^{-1}(z-d)B\end{array}\right]&, \ if \ s>r+1 \\ 
        \ \ \ \ \ \ \ \ \ \ +\gamma\left[\begin{array}{cc}-(b-y)^{-1}(x-a) & 1 \\ 0 & -(b-y)^{-1}(z-d)\end{array}\right]& \end{array}\right.$$
        where $W=A-D-(y-b)(z-d)^{-1}C$.
        \item[(iv)] If $d=z,a\neq x, b \neq y$ and $c\neq w$ then $Z_{rs}$ is equal to
        $$\left\{\begin{array}{ll}\gamma \left[\begin{array}{cc}1 & -(b-y)(x-a)^{-1} \\ -(x-a)^{-1}(c-w) &0\end{array}\right]&, \ if \ s=r+1 \\ 
        &\\
        \left[\begin{array}{cc}0 & (x-a)^{-1}B \\ (y-b)^{-1}W-(x-a)^{-1}(y-b)^{-1}(c-w)B & (c-w)^{-1}C\end{array}\right] &, \ if \ s>r+1, \\ 
        \ \ \ \ \ \ \ \ \ \ +\gamma \left[\begin{array}{cc}1 & -(b-y)(x-a)^{-1} \\ -(x-a)^{-1}(c-w) &0\end{array}\right] & \end{array}\right.$$
        where $W=A-D-(a-x)(c-w)^{-1}C$. 
        \item[(v)] If $a=x,b=y,c\neq w$ and $d\neq z$ then $Z_{rs}$ is equal to
        $$\!\left\{\!\begin{array}{ll}\!\gamma\!\left[\begin{array}{cc}-(z-d)^{-1}(c-w) & 1 \\
        \!-\!(z\!-\!d)^{-1}(c\!-\!w)\beta^{-1}\![z(z\!-\!d)\!^{-1}\!(c\!-\!w)\!+\!w] & \beta^{-1}[z(z-d)^{-1}(c-w)+w]\end{array}\right],& \\
        \hspace{10cm}\!\!\! ,\! if\! \ \!s\!=\!r\!+\!1\!&\\
        &\\
        \!\left[\!\begin{array}{cc}(z-d)^{-1}A & 0 \\ (z-d)^{-1}C-(z-d)^{-1}(c-w)\beta^{-1}W & \beta^{-1}W\end{array}\!\right]\!+\! & \\
        \!\gamma\!\left[\!\begin{array}{cc}-(z-d)^{-1}(c-w) & 1 \\
        \!-\!(z\!-\!d)^{-1}(c\!-\!w)\beta^{-1}[z(z\!-\!d)^{-1}(c\!-\!w)\!+\!w] & \beta^{-1}[z(z-d)^{-1}(c\!-\!w)\!+\!w]\end{array}\!\right]\!,\!&\\
        \hspace{10cm},\! if\! \ \!s\!>\!r\!+\!1\!&
        \end{array}\!\right.$$
        where $W=A+b(z-d)^{-1}C-z(z-d)^{-1}A$ and $\beta=b(z-d)^{-1}(c-w)+a$.
        \item[(vi)] If $a=x, c=w, b \neq y$ and $d \neq z$ then $Z_{rs}$ is equal to:
        \item[$\cdot$] If $w\neq 0$
        $$\left\{\begin{array}{ll}\gamma\left[\begin{array}{cc}0 & w^{-1}a \\ 0 & 1\end{array}\right]&, \ if\ s=r+1 \\
        &\\
        \left[\begin{array}{cc}(b-y)^{-1}B & -w^{-1}(A-z(b-y)^{-1}B+b(z-d)^{-1}C) \\ (z-d)^{-1}C & 0 \end{array}\right] &, \ if \ s>r+1 \\ \ \ \   \ \ \ \ \ \ \ \ \ \ \ \ \ \ \ \ \ \ \ +\gamma\left[\begin{array}{cc}0 & w^{-1}a \\ 0 & 1\end{array}\right] & \end{array}\right.$$
        \item[$\cdot$] If $a\neq 0$ then
        $$\left\{\begin{array}{ll}\gamma\left[\begin{array}{cc}0 & 1 \\ 0 & -a^{-1}w\end{array}\right] &, \ if \ s=r+1 \\
        &\\
        \left[\begin{array}{cc}(b-y)^{-1}B & 0 \\ (z-d)^{-1}C & -a^{-1}(A-z(b-y)^{-1}B+b(z-d)^{-1}C) \end{array}\right]&, \ if \ s>r+1 \\
         \ \ \ \ \ \ \ \ \ \ +\gamma\left[\begin{array}{cc}0 & 1 \\ 0 & -a^{-1}w\end{array}\right] &
        \end{array}\right.$$
        
        \item[(vii)] If $a=x, d=z, b \neq y$ and $c\neq w$ then $Z_{rs}$ is equal to
        $$\left\{\!\begin{array}{ll}\gamma\left[\begin{array}{cc}0 & -(c-w)^{-1}(y-b)\\ 1 & 0\end{array}\right]\!\!\! &\!\!\!,  if\! \ \!s\!=\!r\!+\!1\! \\
        &\\
        \left[\begin{array}{cc}(b-y)^{-1}B & (c-w)^{-1}(A-D) \\ 0 & (c-w)^{-1}C\end{array}\right]+\gamma\left[\begin{array}{cc}0 & -(c-w)^{-1}(y-b)\\ 1 & 0\end{array}\right]\!\!\!&\!\!\!,  if \!\ \!s\!>\!r\!+\!1\!
        \end{array}\!\right.$$
        
        \item[(viii)] If $b=y, c=w, d \neq z$ and $a\neq x$ then $Z_{rs}$ is equal to
        $$\left\{\!\begin{array}{ll}\gamma\left[\begin{array}{cc}-(z-d)^{-1}(a-x) & 0 \\ 0 & 1\end{array}\right]\!\!\!&\!\!\!, if\! \ \!s\!=\!r\!+\!1\! \\
        &\\
        \left[\begin{array}{cc}(z-d)^{-1}(A-D) & (x-a)^{-1}B \\ (z-d)^{-1}C & 0 \end{array}\right]+\gamma\left[\begin{array}{cc}-(z-d)^{-1}(a-x) & 0 \\ 0 & 1\end{array}\right]\!\!\!&\!\!\!,  if \!\ \!s\!>\!r\!+\!1\! \end{array}\!\right.$$
        \item[(ix)] If $b=y, d=z, c \neq w$ and $a\neq x$ then $Z_{rs}$ is equal to
        \item[$\cdot$] If $z\neq 0$:
        $$\left\{\begin{array}{ll}\gamma\left[\begin{array}{cc}z^{-1}b & 0 \\ 1 & 0\end{array}\right] &,\ if \ s=r+1 \\
        &\\
        \left[\begin{array}{cc}z^{-1}(A+w(x-a)^{-1}B-a(c-w)^{-1}D) & (x-a)^{-1}B \\ 0 & (c-w)^{-1}D\end{array}\right]&, \ if \ s>r+1 \\ 
         \ \ \ \ \ \ \ \ \ \ +\gamma\left[\begin{array}{cc}z^{-1}b & 0 \\ 1 & 0\end{array}\right]&\end{array}\right.$$
        \item[$\cdot$] If $b \neq 0$:
        $$\left\{\begin{array}{ll}\gamma\left[\begin{array}{cc}1 & 0 \\ b^{-1}z& 0\end{array}\right] &, \ if \ s=r+1 \\
        &\\
        \left[\begin{array}{cc}0 & (x-a)^{-1}B \\ -b^{-1}(A+w(x-a)^{-1}B-a(c-w)^{-1}D)& (c-w)^{-1}D\end{array}\right]&, \ if \ s>r+1 \\ \ \ \ \ \   \ \ \ \ \ \ \ \ \ \ +\gamma\left[\begin{array}{cc}1 & 0 \\ b^{-1}z& 0\end{array}\right]&\end{array}\right.$$
        
        \item[(x)] If $c=w, d = z, b\neq y$ and $a\neq x$ then $Z_{rs}$ is equal to
        $$\left\{\!\begin{array}{ll}\!\gamma\!\left[\begin{array}{cc}(b\!-\!y)^{-1}\beta^{-1}(x\!-\!a)(b(y\!-\!b)^{-1}(a\!-\!x)\!+\!w) & -\beta^{-1}(b(y\!-\!b)^{-1}(a\!-\!x)\!+\!w) \\ -(y\!-\!b)^{-1}(a\!-\!x) & 1\end{array}\right]& \\
         \hspace{10cm}, if \!\ \!s\!=\!r\!+\!1\!&\\
         &\\
        \left[\begin{array}{cc}(b-y)^{-1}(B-(x-a)^{-1}\beta^{-1}W) & \beta^{-1}W \\ (y-b)^{-1}(A-D) & 0\end{array}\right]&  \\
          +\!\gamma\!\left[\!\begin{array}{cc}(b\!-\!y)^{-1}\beta^{-1}(x\!-\!a)(b(y\!-\!b)^{-1}(a\!-\!x)\!+\!w) & -\beta^{-1}(b(y\!-\!b)^{-1}(a\!-\!x)\!+\!w) \\ -(y\!-\!b)^{-1}(a\!-\!x) & 1\end{array}\!\right] &
        \\
        \hspace{10cm} ,if \ s\!>\!r\!+\!1\!&
        \end{array}\!\right.$$
   where $W=A-z(b-y)^{-1}B+b(y-b)^{-1}(A-D)$ and $\beta=-z(b-y)^{-1}(x-a)-w$.
        \item[(xi)] If all entries of $Z_{rr}$ and $Z_{ss}$ are different, then $Z_{rs}$ is equal to
        $$\left\{\begin{array}{ll}\gamma\left[\begin{array}{cc}(x-a)(z-d)^{-1} & - (z-d)^{-1}(b-y)
    \\
    -(z-d)^{-1}(c-w) & 1\end{array}\right]\!\!\! &\!\!\!,  if \! \ \!s\!=\!r\!+\!1\! \\
    &\\
    \left[\begin{array}{cc}(b-y)^{-1}[B+(x-a)(z-d)^{-1}\theta^{-1}W] & -(z-d)^{-1}\theta^{-1}W \\ (z-d)^{-1}C & 0\end{array}\right]+\!\!\!&\!\!\!,  if \ \!s\!>\!r\!+\!1\! \\
    \ \ \ \ \ \ \ \ + \gamma\left[\begin{array}{cc}(x-a)(z-d)^{-1} & - (z-d)^{-1}(b-y)
    \\
    -(z-d)^{-1}(c-w) & 1\end{array}\right]&
    \end{array}\right.$$
     where $W=(z-d)[(b-y)(A-D)-(z-d)B]-(y-c)C$ and, in all of the above cases, $\gamma\in \mathbb{Z}_p$ can be an arbitrary element.
    \end{itemize}
\end{itemize}
\end{theorem}
\begin{proof} 
Whitout lost of generality, denote by 
$$Z_{rr}=\left[\begin{array}{cc} a & b \\ c & d\end{array}\right],\ \  Z_{rs}=\left[\begin{array}{cc} z_1 & z_2 \\ z_3 & z_4\end{array}\right] \ \ \ and \ \  \ Z_{ss}=\left[\begin{array}{cc} x & y \\ w & z\end{array}\right].$$So we have that
$$\widetilde{Z}_{rr}=\left[\begin{array}{cc} d & -b \\ -c & a\end{array}\right], \ \ \widetilde{Z}_{rs}=\left[\begin{array}{cc} z_4 & - z_2 \\ -z_3 & z_1\end{array}\right] \ \ \ and \ \ \ \widetilde{Z}_{ss}=\left[\begin{array}{cc} z &- y \\ - w & x\end{array}\right].$$

\begin{itemize} 
    \item [1.)] Follows inmediately from Lemma \ref{lemma1}.
    \item[2.)] If $s=r+1$, in this case we obtain 
    $$Z_{rr}\widetilde{Z}_{r,r+1}+Z_{r,r+1}\widetilde{Z}_{r+1,r+1}=0.$$
    So, we obtain the system
    \begin{equation}\label{system1}
    \left[\begin{array}{cccc}z & -w & -b & a \\ b-y & x-a & 0 & 0 \\ 0 & 0 & z-d & c-w \\ d & -c & -y & x\end{array}\right]\left[\begin{array}{c}z_1 \\ z_2 \\ z_3 \\ z_4\end{array}\right]= \left[\begin{array}{c}0 \\ 0 \\ 0 \\ 0 \end{array}\right].
    \end{equation}
    In the case that $s>r+1$  we obtain 
    $$Z_{rr}\widetilde{Z}_{rs}+Z_{r,r+1}\widetilde{Z}_{r+1,s}+\cdots+Z_{r,s-1}\widetilde{Z}_{s-1,s}+Z_{rs}\widetilde{Z}_{ss}=0$$
    or
    \begin{equation}\label{equation2}
    Z_{rr}\widetilde{Z}_{rs}+Z_{rs}\widetilde{Z}_{ss}=-\sum^{s-1}_{t=r+1}Z_{rt}\widetilde{Z}_{ts}.
    \end{equation}
    By hipothese, denote by $\displaystyle-\sum^{s-1}_{t=r+1}Z_{rt}\widetilde{Z}_{ts}=\left[\begin{array}{cc}A & B \\ C & D\end{array}\right],$
    then from equation (\ref{equation2}) we obtain the following system
    \begin{equation}\label{system2}
    \left[\begin{array}{cccc}z & -w & -b & a \\ b-y & x-a & 0 & 0 \\ 0 & 0 & z-d & c-w \\ d & -c & -y & x\end{array}\right]\left[\begin{array}{c}z_1 \\ z_2 \\ z_3 \\ z_4\end{array}\right]= \left[\begin{array}{c}A \\ B \\ C \\ D \end{array}\right]. 
    \end{equation}
    \item[2a.)] If $Z_{rr}=Z_{r+1,r+1}$ then from the system (\ref{system1}) we have  the equation
    $$dz_1-cz_2-bz_3+az_4=0.$$
    How $\det(Z_{rr})=1$ then some entry $a,b,c$ or $d$ is not zero, from this, for example, if $d\neq 0$ then we can write
    $$z_1=d^{-1}(cz_2+bz_3-az_4).$$
    From this
    $$Z_{r,r+1}=z_2\left[\begin{array}{cc}d^{-1}c & 1 \\ 0 & 0\end{array}\right]+z_3\left[\begin{array}{cc}d^{-1}b & 0 \\ 1 & 0\end{array}\right]+z_4\left[\begin{array}{cc}-d^{-1}a & 0 \\ 0 & 1\end{array}\right],$$
    with arbitrary $z_2,z_3,z_4 \in \mathbb{Z}_p$. 
     If $Z_{rr}=Z_{ss}$ with $s>r+1$ then, from the system (\ref{system2}), necessarily $B=C=0$ and $A=D$. Here we obtain the equation
     $$dz_1-cz_2-bz_3+az_4=A.$$
     Similarly to the above item, we can consider $d\neq 0$, then $$z_1=d^{-1}(A+cz_2+bz_3-az_4).$$
     Thus
     $$Z_{rs}=z_2\left[\begin{array}{cc}d^{-1}c& 1 \\ 0 & 0\end{array}\right]+z_3\left[\begin{array}{cc}d^{-1}b& 0 \\ 1 & 0\end{array}\right]+z_4\left[\begin{array}{cc}-d^{-1}a& 0 \\ 0 & 1\end{array}\right]+\left[\begin{array}{cc}d^{-1}A& 0 \\ 0 & 0\end{array}\right],$$
     where $z_2,z_3,z_4\in \mathbb{Z}_p$ are arbitrary elements.
\item[2b.)] In the case that $Z_{rr}\neq Z_{ss}$ we study the following cases:
        \item [(i)] If $a=x,b\neq y, c\neq w$ and $d\neq z$. Then from the  equation (\ref{equation2}) we obtain that $(b-y)z_1=B$ and from here $z_1=(b-y)^{-1}B.$ From the first and the fourth equation of the system (\ref{system2}) we obtain
        $$(z-d)z_1+(c-w)z_2+(y-b)z_3=A-D$$
        or
        \begin{equation}\label{equation5}
            (c-w)z_2+(y-b)z_3=A_1 
        \end{equation}
        with $A_1=A-D-(z-d)(b-y)^{-1}B.$
        From equation (\ref{equation5}) and the third equation of the system (\ref{system2}) we obtain the new following system
        $$\left\{\begin{array}{cl} (z-d)(c-w)z_2+(z-d)(y-b)z_3=&(z-d)A_1 \\
        (y-b)(z-d)z_3+(y-b)(c-w)z_4=&(y-b)C.\end{array}\right.,$$
        and substraying  both equation, we obtain
        $$(z-d)(c-w)z_2-(y-b)(c-w)z_4=(z-d)A_1-(y-b)C.$$
        then
        $$z_2=(z-d)^{-1}(c-w)^{-1}\left\{(z-d)A_1-(y-b)C+(y-b)(c-w)z_4\right\}.$$
        Also, from the third equation of the system (\ref{system2})  we obtain
        $$z_3=(z-d)^{-1}\left\{ C-(c-w)z_4\right\}.$$
        So
        $$
        \begin{array}{rcl}
        Z_{rs}&=&z_4\left[\begin{array}{cc} 0 & (z-d)^{-1}(y-b) \\ -(z-d)^{-1}(c-w) & 1\end{array}\right]+\\
        &&\\
        &&+\left[\begin{array}{cc} (b-y)^{-1}B & (c-w)^{-1}A_1-(z-d)^{-1}(c-w)^{-1}(y-b)C \\ (z-d)^{-1}C & 0\end{array}\right]
        \end{array}
        $$
        for arbitrary $z_4 \in \mathbb{Z}_p$. The cases (ii),(iii) and (iv) solves similarly.
        
        \item[(v)] In this case we have that $a=x,b=y,c\neq w$ and $d\neq z$, then from system (\ref{system2}), we have $B=0$. From first and fourth equations into the same system we obtain
        \begin{equation}z_1=(z-d)^{-1}[A-(c-w)z_2],\label{z1}\end{equation}
        and from second equation of the same system we obtain
        \begin{equation}z_3=(z-d)^{-1}[C-(c-w)z_4].\label{z3}\end{equation}
        Substituting the above relations into first equation of the system (\ref{system2}) we obtain
        \begin{equation}-[z(z-d)^{-1}(c-w)+w]z_2+[b(z-d)^{-1}(c-w)+a]z_4=E,\label{equabeta}\end{equation}
        with $E=A+b(z-d)^{-1}C-z(z-d)^{-1}A.$ Now, analyzing the term $\beta=b(z-d)^{-1}(c-w)+a$. Suppose that $\beta=0$ then
        $$b(z-d)^{-1}(c-w)+a=0$$
        or
        \begin{equation}\label{equaresul} az-ad+bc-bw=0.\end{equation}
        By item (1) of our proof, we obtain that $\det(Z_{rr})=ad-bc=1$. Then substituting in the equation $(\ref{equaresul})$ we have
        \begin{equation}az-bw=1,\label{equanew}\end{equation}        
        but, by hipotesse $Z_{rr}\neq Z_{ss}$ then $Z_{rr}\widetilde{Z}_{ss}=\left[\begin{array}{cc}az-bw & * \\ * & *\end{array}\right]\neq I_2$, in particular $az-bw\neq 1$ and this contraries the equation (\ref{equanew}). So $\beta\neq 0$ and by the equation $(\ref{equabeta})$ we obtain that
        \begin{equation}z_4=\beta^{-1}[E+[z(z-d)^{-1}(c-w)+w]z_2].\label{z4}\end{equation}
        From the relations (\ref{z1}), (\ref{z3}) and (\ref{z4}) we finally obtains that $Z_{rs}$ is equal to
        $$\left[\begin{array}{cc}(z-d)^{-1}A & 0 \\ (z-d)^{-1}C-(z-d)^{-1}(c-w)\beta^{-1}E & \beta^{-1}E\end{array}\right]+$$ 
        $$\ \ \ \ +z_2\left[\!\begin{array}{cc}-(z\!-\!d)^{-1}(c\!-\!w) & 1 \\
        -(z\!-\!d)^{-1}(c\!-\!w)\beta^{-1}[z(z\!-\!d)^{-1}(c\!-\!w)\!+\!w] & \beta^{-1}[z(z\!-\!d)^{-1}(c\!-\!w)\!+\!w]\end{array}\!\right],$$
        where $E=A+b(z-d)^{-1}C-z(z-d)^{-1}A$ and $\beta=b(z-d)^{-1}(c-w)+a$. So, $Z_{rs}$ depending only one variable.
        \item[(vi)] In this case, from the second equation of the system (\ref{system2}) we obtain
        $$z_1=(b-y)^{-1}B,$$ and from the thrid equation
        $$z_3=(z-d)^{-1}C,$$
        and substuting in the first equation
        $$zz_1-wz_2-bz_3+az_4=A.$$
        If $w\neq 0$, we obtain 
        $$z_2=-w^{-1}(A-zz_1+bz_3-az_4),$$
        and from here, $Z_{rs}$ is equal to:
        $$\left[\begin{array}{cc}(b-y)^{-1}B & -w^{-1}(A-z(b-y)^{-1}B+b(z-d)^{-1}C) \\ (z-d)^{-1}C & 0 \end{array}\right]+z_4\left[\begin{array}{cc}0 & w^{-1}a \\ 0 & 1\end{array}\right],$$
        and if $a\neq 0$, we obtain
        $$z_4=-a^{-1}(A-zz_1wz_2+bz_3),$$
        and from here, $Z_{rs}$ is equal to
        $$ \left[\begin{array}{cc}(b-y)^{-1}B & 0 \\ (z-d)^{-1}C & -a^{-1}(A-z(b-y)^{-1}B+b(z-d)^{-1}C) \end{array}\right]+z_2\left[\begin{array}{cc}0 & 1 \\ 0 & -a^{-1}w\end{array}\right].$$
        \item[(vii)] If $a=x,d=z,b\neq y $ and $c\neq w$ then, from the system we have immediately that
        \begin{eqnarray}
        z_1&=&(b-y)^{-1}B,\label{az1}\\ 
        z_4&=&(c-w)^{-1}C,\label{az4}\\
        z_2&=&(c-w)^{-1}[A-D-(y-b)z_3]. \label{az2}\end{eqnarray}
        So $Z_{rs}$ is equal to:
        $$\left[\begin{array}{cc}(b-y)^{-1}B & (c-w)^{-1}(A-D) \\ 0 & (c-w)^{-1}C\end{array}\right]+z_3\left[\begin{array}{cc}0 & -(c-w)^{-1}(y-b)\\ 1 & 0\end{array}\right]$$
        where $z_3 \in \mathbb{Z}_p$ and here, $Z_{rs}$ depending only one variable. 
        \item[(viii)] Solves similarly that (vii).
        \item[(ix)] Solves similarly that (vi).
        \item[(x)] Solves similarly that item (v).
        \item[(xi)] In this case $Z_{rr}\neq Z_{ss}$, $a\neq x, b\neq y, c\neq w$ and $d\neq z$. 
    
    From first and fourth equation of the system (\ref{system2}) we obtain
    \begin{equation}(z-d)z_1+(c-w)z_2+(y-b)z_3+(a-x)z_4=A-D\label{equation6}  \end{equation}
    \begin{equation}(b-y)z_1+(x-a)z_2 =  B \label{equation7} \end{equation}
    \begin{equation}(z-d)z_3+(c-w)z_4= C\label{equation8} \end{equation}
    From the equations (\ref{equation6}) and (\ref{equation7}) we have
    \begin{equation}
        [(c\!-\!w)(b\!-\!y)-(z\!-\!d)(x\!-\!a)]z_2+(b\!-\!y)(y\!-\!b)z_3+(b\!-\!y)(a\!-\!x)z_4=(b\!-\!y)(A\!-\!D)-(z\!-\!d)B\label{equation9}
    \end{equation}
    and from equations (\ref{equation9}) and (\ref{equation8}) we have
    \begin{equation}
        (z-d)[(c-w)(b-y)-(z-d)(x-a)]z_2+[(a-x)(b-y)(z-d)-(y-b)(b-y)(c-w)]z_4=E \label{equation10}
    \end{equation}
    where $E=(z-d)[(b-y)(A-D)-(z-d)B]-(y-b)C$,
    or, from the last equation we simplify and obtain
    \begin{equation}
        [(x-a)(z-d)-(c-w)(b-y)]\cdot[(z-d)z_2+(b-y)z_4]=-E \label{equationA}
    \end{equation}
    
    \begin{claim}
    $\theta=(x-a)(z-d)-(c-w)(b-y)\neq 0$
    \end{claim}
    Notice that
    $\theta=\det\left(\left[\begin{array}{cc} x & y \\ w & z\end{array}\right]+\left[\begin{array}{cc}-a & -b \\ -c & - d\end{array}\right]\right)=\det(Z_{ss}+(-Z_{rr}))$ and use the following formula for $2\times 2$ matrices:
    $$\det(A+B)=\det A+\det B + \det A\cdot Trace(A^{-1}B).$$
    and observing that
    $$\begin{array}{ll}
         Trace(Z_{ss}^{-1}(-Z_{rr}))&=Trace(\left[\begin{array}{cc}z & -y \\ -w & x\end{array}\right]\left[\begin{array}{cc}- a & -b \\ -c & -d\end{array}\right])  \\
         &= Trace\left(\left[\begin{array}{cc}-za+yc & * \\ * &  bw-xd\end{array}\right]\right) \\
         & =(yc-az)+(bw-xd).
    \end{array}
    $$
    Substitying in the above formula
    $$\theta=2+(yc-az)+(bw-xd)$$
    How $Z_{rr}\neq Z_{ss}$ then $Z_{rr}Z_{ss}^{-1}=Z_{rr}\widetilde{Z}_{ss}\neq I_2$ and for this we obtain that $xd-cy \neq 1$ and $az-bw\neq 1$ and for this we conclude that $\theta \neq 0.$
    So, follows from equation (\ref{equationA}) that 
    $$z_2=(z-d)^{-1}[-\theta^{-1}E-(b-y)z_4].$$
    So
    $$z_1=(b-y)^{-1}\left[B-(x-a)(z-d)^{-1}[-\theta^{-1}E-(b-y)z_4]\right],$$
    and $$z_3=(z-d)^{-1}[C-(c-w)z_4].$$
    So, $Z_{rs}$ is equal to:
    $$ \ \ \ \left[\begin{array}{cc}(b-y)^{-1}[B+(x-a)(z-d)^{-1}\theta^{-1}E] & -(z-d)^{-1}\theta^{-1}E \\ (z-d)^{-1}C & 0\end{array}\right]+$$
    $$\ \ \ \ \ \ \ \ \ \ \ \ \ \ \ \ \ \ \ \ \ \ \ \  +z_4\left[\begin{array}{cc}(x-a)(z-d)^{-1} & - (z-d)^{-1}(b-y)
    \\
    -(z-d)^{-1}(c-w) & 1\end{array}\right]  $$ 
    with $z_4\in \mathbb{Z}_p$ an arbitrary element.
    \end{itemize}
    \end{proof}
\begin{remark}
Using the Theorem \ref{teo4} and appliying the isomorphims inverse $\varphi^{-1}$ we can construct coninvolution matrices over $T_{n}(\mathbb{Z}[i,j,k])$ considering each case of the theorem. 
Describe this is too long and we leave it to the reader 
\end{remark}
And, with this we can compute the number of  Coninvolutions in $T_n(\mathbb{Z}_p[i,j,k])$:
\begin{proof}[Proof of Theorem \ref{th3}]
By the Lemma \ref{lemma1} we know that there are $(p^2-1)p$ of solutions of $z\bar{z}=1$, then we enumerate all solutions and denote $q_1,q_2,\cdots, q_s$ these solutions. Denote by $n_1,n_2,\cdots, n_s$ the number of times that $q_1,q_2,\cdots,q_s$ appears in the diagonal respectively, such that $n_1+n_2+\cdots+ n_s=n$ and $n_i\geq 0$ for all $i=1,2,\cdots,s$. Observe that the value $n_j=0$ means that the solution $q_j$ does not appears in the diagonal of matrix. Then, by the Theorem \ref{teo4} we conclude that, if $q_r$ and $q_s$ appears into the diagonal $n_r$ and $n_s$ times, respectively, then we have that count:
\begin{itemize}
    \item  $p^{3\frac{n_r(n_r-1)}{2}}\times p^{3\frac{n_s(n_s-1)}{2}}$, this because there are $n_r$ diagonals that are equals and similarly there are $n_s$ equal entries in the diagonal.
    \item $p^{n_r\cdot n_s}$ because there are $n_r$ and $n_s$ different entries in the diagonal. 
\end{itemize}
And the combinatory number means the choose such that we can put the solutions $q_1,q_2,\cdots,q_s$ into the diagonal of a matrix. With this we conclude that $\mathcal{CI}(n,\mathbb{Z}_p[i,j,k])$ is equal to $$\sum^{n}_{\substack{n_1+n_2+\cdots+n_s=n \\ n_1\geq 0, n_2\geq, \cdots, n_s\geq 0 }}\binom{n}{n_1n_2\cdots n_s}  p^{\displaystyle\sum_{1\leq i<j \leq s}n_in_j}\cdot p^{3\displaystyle\sum_{i=1 \ \wedge n_i>1}^s \frac{1}{2}(n_i(n_i-1))}.$$
Notice that $\sum_{1\leq i < j \leq s} n_in_j= \frac{1}{2}\left\{(\sum^{s}_{i=1}n_i)^2-\sum^s_{i=1}n_i^2\right\}= \frac{1}{2}\left\{n^2-\sum^s_{i=1}n_i^2\right\},$
and using the Slowik's formula (see \cite{Slowik2}) we can simplify the above formula for  $\mathcal{CI}(n,\mathbb{Z}_p[i,j,k])$. 
\end{proof}
\section{Tables}

\begin{table}[H]
\caption{Number of coninvolutions in the group $T_n(\mathbb{Z}_p[i])$}
\centering
\begin{tabular}{c|c|c|c|c|c}\hline\hline
$p$ & $|U(\mathbb{Z}_p[i])|$&$n=2$ &  $n=3$ & $n=4$ & $n=5$ \\ 
\hline
2 &2& 8&  64& 1024 &32768 \\  
3 &8& 192& 13824 & 2985984  &$32768\times 3^{10}$ \\ 
4 &8&256 &32768 & 16777216 &$32768\times 4^{10}$ \\
5 & 16&1280 & 512000 & $65536\times 5^{6}$  &$1048576\times 5^{10}$\\
6 & 16&1536 &884736 &  $65536\times 6^{6}$& $1048576\times 6^{10}$\\
7 & 48 &4816128 & 37933056 & $5308416\times 7^{6}$ & $48^{5}\times 7^{10}$\\
8 & 32&8192&16777216  &$1048576\times 8^{6}$  & $32^{5}\times 8^{10}$\\
9 & 72&46656& 272097792 &  $26873856\times 9^{6}$&  $72^{5}\times 9^{10}$\\
10& 32&19240& 32768000 & $1048576\times 10^{6}$ &  $32^{5}\times 10^{10}$\\
11& 120&158400 & $120^3\times 11^3$ & $120^{4}\times 11^{6}$ & $120^{5}\times 11^{10}$\\[1ex]
\hline
\end{tabular}
\label{}
\end{table}

\begin{table}[H]
\caption{Number of coninvolutions in the group $T_n(\mathbb{Z}_p[i,j,k])$}
\centering
\begin{tabular}{c|c|c|c}\hline\hline
$p$ &    $n=3$ & $n=4$ & $n=5$ \\ 
\hline
3 &  730944 & 1935197568 & 96505626807936 \\  
5 &  55921200 & 242942845440 & 5540326592171520\\ 
7 & 1097237568 & 10865680662144 & 348938732224113024\\
11 & 63229077120& 2297513874543360 &252342301214307575040\\
13 &  284363059344& 16960852265255808 & 3047262840713675665536\\
17 &  3184303946688 & 423042785745341184 & 168892742231459264058624\\
19 &  8670723975360 & 1607003347848750720 & 894581852305833432708480\\
23 & 48451556497344 & 15918297951624613248 & 15699903765051877402475136 \\
29& 390682469205840 & 257213522072770327680 & 508193153533449590679553920\\
[1ex]

\hline
\end{tabular}
\label{}
\end{table}

\end{document}